\documentclass[11pt]{amsart}
\usepackage{graphicx}
\usepackage{amssymb}
\usepackage[usenames,dvipsnames]{color}
\numberwithin{equation}{section} 

\newtheorem{theorem}{Theorem}[section]
\newtheorem{lemma}[theorem]{Lemma}

\newtheorem{corollary}[theorem]{Corollary}

\newtheorem{remark}[theorem]{Remark}
\newtheorem{definition}{Definition}[section]

\def\R{{\mathbb R}}

\def\N{{\mathbb N}}

\def\cL{{\mathcal L}}
\def\cM{{\mathcal M}}

\def\a{\alpha}
\def\b{\beta}
\def\e{\varepsilon}
\def\d{\delta}
\def\D{\Delta}

\def\G{\Gamma}
\def\k{\kappa}
\def\l{\lambda}
\def\L{\Lambda}
\def\m{\mu}
\def\n{\nabla}
\def\p{\partial}
\def\r{\rho}
\def\s{\sigma}
\def\t{\tau}
\def\w{\omega}
\def\W{\Omega}

\def\1{\left(}
\def\2{\right)}
\def\3{\left\{}
\def\4{\right\}}
\def\8{\infty}
\def\sm{\setminus}
\def\ss{\subseteq}

\DeclareMathOperator*{\diam}{diam}
\DeclareMathOperator*{\supp}{supp}

\title{Regularity for solutions of non local parabolic equations}
\author[H. Chang Lara]{H\'ector Chang Lara}
\address{%
University of Texas at Austin\\
Department of Mathematics\\
1 University Station C1200\\
Austin, TX 78712
}
\email{hchang@math.utexas.edu}
\author[G. D\'avila]{Gonzalo D\'avila}
\address{%
University of Texas at Austin\\
Department of Mathematics\\
1 University Station C1200\\
Austin, TX 78712
}
\email{gdavila@math.utexas.edu}

\begin{document}
\begin{abstract}
We study the regularity of solutions of parabolic fully nonlinear nonlocal equations. We proof $C^\a$ regularity in space and time and, under different assumptions on the kernels, $C^{1,\a}$ in space for translation invariant equations. The proofs rely on a weak parabolic ABP and the classic ideas of \cite{T} and \cite{W}. Our results remain uniform as $\s\to2$ allowing us to recover most of the regularity results found in \cite{W}.
\end{abstract}

\maketitle


\section{Introduction}\label{Introduction}

In this paper we are interested in studying regularity for solutions of 
\begin{align}\label{eqbase}
u_t-Iu=0,
\end{align}
where $I$ is a fully nonlinear nonlocal operator. The previous type of equations arises naturally when studying evolution equations for purely discontinuous L\'evy processes. In this case $I$ is given by the jump part (which is a linear operator, usually denoted by $L$), i.e. for $u \in C^\8_0$,
\begin{align}\label{linear}
Lu(x,t) &= P.V.\int (u(x+y,t)-u(x,t)-\nabla u(x,t)\cdot y\chi_{B_1}(y))d\mu(y)\\
\nonumber & = \lim_{\e\to 0} \int_{\R^n\sm B_\e}(u(x+y,t)-u(x,t)-\nabla u(x,t)\cdot y\chi_{B_1}(y))d\mu(y).
\end{align}
The type measures we are considering are of the form $d\m = K(x,t;y)dy$ for some kernel positive $K$ even in the $y$ variable, i.e. $K(x,t;y) = K(x,t;-y)$. This assumption allows us to rewrite the operator without the principal value in the following way,
\begin{align*}
Lu(x,t) = \int_{\R^n}\d(u,x,t;y)K(x,t;y)dy,
\end{align*}
where $2\d(u,x,t;y)=u(x+y,t)+u(x-y,t)-2u(x,t)$ is the second order difference in space of $u$ at $(x,t)$.

More complicated equations appear in competitive games, in this case the nonlocal operator is given by and inf sup or sup inf combination of linear operators,
\begin{align}
\label{eqinfsup} \text{(Inf-sup type) } Iu(x,t) &= \inf_\b\sup_\a L_{\a,\b} u(x,t),\\
\label{eqsupinf} \text{(Sup-inf type) } Iu(x,t) &= \sup_\b\inf_\a L_{\a,\b} u(x,t).
\end{align}
A particular case of the previous ones are the extremal operators which we denote in the following way
\begin{align}
\label{moperator1} \text{(Maximal) } \cM^{+}_{\cL}u(x,t)&=\sup_{L\in\cL}Lu(x,t),\\
\label{moperator2} \text{(Minimal) } \cM^{-}_{\cL}u(x,t)&=\inf_{L\in\cL}Lu(x,t).
\end{align}
In general, we are interested in operators satisfying a standard ellipticity condition, which allow us to control $I(u+v)-Iu$ by extremal operators. We will give a precise definition in Section \ref{DAP}. 

In the case of a family of symmetric kernels $K_{\a,\b}(y)=a_{\a,\b}(y)|y|^{-(n+1)}$, L. Silvestre studied in \cite{S2} the regularity of the solution of the Hamilton-Jacobi equation
\begin{align*}
u_t-\sup_\a\inf_\b\1c_{\a,\b}+b_{\a,\b}\cdot\n u+\int\d(u,x,y)\frac{a_{\a,\b}(y)}{|y|^{n+1}}dy\2=0,
\end{align*}
where $c_{\a,\b}$ is a family of constants, $b_{\a,\b}$ is a bounded family of vectors and $\l\leq a_{\a,\b}\leq \L$. He was able to prove that the solution of the equation was classic using a non-variational approach to proof a diminish of oscillation lemma. A modification of the proof in \cite{S2} allows to get the regularity for for the same type of equations we study here, but the estimates would blow up as $\s\to 2$.

The variational problem was studied by L. Caffarelli, C. Chan and A. Vasseur in \cite{CHV} by using De Giorgi's technique. Also recently,  M. Felsinger and M. Kassmann in \cite{FK}, obtained a Harnack inequality where the constants remain uniform as the order of the equation goes to the classical one by using Moser's technique.  

The focus of this paper is to study regularity of solutions of \eqref{eqbase}, that remain uniform as $\s\to 2$. This will provide a natural extension to part of the theory already developed by L. Wang in \cite{W}.

The paper is divided as follows. In Section \ref{DAP} we gave the proper definitions of viscosity solutions and maximal operators. We will also specify what type of equations we are dealing with, write the hypothesis over the kernels and give some properties. For instance, we state there a comparison principle and the existence of solutions for the Dirichlet problem. Section \ref{SectionABP} is the heart of this article. There we proof a very weak ABP type of estimate that will allow us to proof a point estimate in Section \ref{SPE}. The strategy to prove the ABP consist in showing that we can cover the contact set of $u$ (with respect to its convex envelope $\G$) by a set where $u$ does not separate too much from $\G$ in a fixed portion. This will consist of two parts, controlling the convex envelope by above and below around a contact point respectively. For the control by above, we iterate Lemma 5.1 in \cite{S2}, for the second part we use an appropriated barrier to show that $\G$ does not decrease too fast. Section \ref{SPE} deals with the point estimate which follows the ideas of \cite{W}. Finally in Section \ref{RR} we state and proof our regularity results.


\section{Definitions and Preliminaries}\label{DAP}


\subsection{Non local operators}

To be precise about the formulas we presented in the previous section, we need to ask and integrability condition to $K$ around the origin,
\begin{align}\label{intcond}
\int_{B_1} |y|^2 K(x,t;y)dy < \infty.
\end{align}
It allows us to write rigorously
\begin{align*}
Lu(x,t) = \int \d(u,x,t;y)K(x,t;y)dy
\end{align*}
not only when $u$, with compact support, is in $C^\8_0$, but also when $u(\cdot,t) \in C^{1,1}(x)$.

\begin{definition}
$u(\cdot,t) \in C^{1,1}(x)$, if there exists a vector $v \in \R^n$ and a number $M>0$ such that
\begin{align*}
|u(y+x,t)-u(x,t)-v\cdot y| < M|y|^2 \text{ for $|y|$ small enough.}
\end{align*}
\end{definition}

Notice that this definition implies $|\d(u,x,t;y)| = O(|y|^2)$ as $|y|$ is close to zero. This is why we can get rid of the principal value in the integral.

We say that a family $\cL$ of linear operators satisfy the integrability condition uniformly in $\W\times[-T,0]$ when the upper bounds in \eqref{intcond} can be taken independent of $L \in \cL$ and $(x,t) \in \W\times[-T,0]$.

We will consider as in \cite{C2} absolute continuous weights $\w$ which measure the contributions of the tails to the non local operators. This allows to compute $Lu$ even when $u$ does not have compact support.

\begin{definition}
The space of function $L^1(\w)$ consist of all $u:\R^n \to \R$ such that
\begin{align*}
\|u\|_{L^1(\w)} := \int |u(y)|\w(y)dy < \8.
\end{align*}
\end{definition}

\begin{definition}[Non local operators]\label{nonlocalop}
We say that $I$ is a non local, fully non linear operator with respect to $\w$, if for every $u(\cdot,t) \in C^{1,1}(x) \cap L^1(\w)$, $Iu(x,t)$ is a well defined real number.
\end{definition}

For $\s \in (0,2)$ fixed, linear operators with kernels of the form $K \sim |y|^{-(n+\s)}$ or combinations of those (by taking supremums and infimums) are contained by this definition. In such cases we say that the operator have order $\s$ and use $\w = 1/(1+|y|^{n+\s})$.  

We say that $I$ is translation invariant in space if
\begin{align*}
I\t_{(x-y,0)}u(y,s) = Iu(x,t),
\end{align*}
where $\t$ is the shift operator,
\begin{align*}
\t_{(x,t)}u(y,s) = u(y+x,s+t).
\end{align*}


\subsection{Continuous operators}

For time dependent problems, the natural topology to use in $\R^n\times\R$ is the so called parabolic topology. It is generated by neighborhoods of the form $B_r(x)\times(t,t-\t]$, for a given point $(x,t) \in \R^n\times\R$. For instance, in this topology a function $f$ is continuous if and only if $f(y,s) \to f(x,t)$ as $(y,s)\to(x,t^-)$.

Whenever we want to see that a linear operator $L$ has ``continuous coefficients" we fix a smooth test function $u$ in an open set $O\ss \R^n\times\R$ and check if $Lu$ evaluated at $O$ is continuous. In the non local case, we need to use not only functions which are smooth in $O$ but also that the contributions from their tails vary in a sufficiently smooth way. This is the motivation to introduce the following space.

\begin{definition}
Let $C(a,b;L^1(\w))$ be the space of function $u:(a,b]\to\R$ such that
\begin{enumerate}
\item for every $t\in(a,b]$, $u(\cdot,t)\in L^1(\w)$,
\item for every $t_2 \in (a,b]$, $\|u(\cdot,t_1)-u(\cdot,t_2)\|_{L^1(\w)} \to 0$ as $t_1\to t_2^-$.
\end{enumerate}
It comes additionally with the norm,
\begin{align*}
\|u\|_{C(a,b;L^1(\w))} = \sup_{t\in(a,b]}\|u(\cdot,t)\|_{L^1(\w)}.
\end{align*}
\end{definition}

The space of functions against which we test the continuity of $I$ are given by parabolic second order polynomials and functions in $C(a,b;L^1(\w))$.

\begin{definition}[Test functions]
The space $S = S(\W\times(-T,0])$ of test functions is the set of all pairs $(v,B_r(x)\times(t-\t,t])$ such that $v \in C(t-\t,t;L^1(\w))$, $B_r(x)\times(t-\t,t] \ss \W\times(-T,0]$ and $v$ restricted to $B_r(x)\times(t-\t,t]$ is a quadratic parabolic polynomial, i.e.
\begin{align*}
v(x,t) = \sum_{i,j=1}^n a_{i,j}x_ix_j + \sum_{i=1}^n b_ix_i + ct + d. 
\end{align*}
\end{definition}

\begin{definition}[Continuous operators]
We say that a non local operator $I$, with respect to $\w$, depends continuously on the position in $\W\times(-T,0]$ if for every $(v,B_r(x)\times(t-\t,t]) \in S$, we have that $Iu$ is a continuous function in $B_r(x)\times(t-\t,t]$ (with respect to the parabolic topology).
\end{definition}

We can understand a little bit better how the space $C(a,b;L^1(\w))$ appears as a requirement for the continuity of the operator in time. Without this condition, even the fractional laplacian would not be a continuous operator in $B_1\times(-1,0]$ with respect to any positive $\w$. Take for example $u$ equal to zero in $B_1\times(-1,0]$ and let vary $u$ freely outside $B_1\times(-1,0]$.


\subsection{Ellipticity.}

In the classical stationary case ellipticity means that, for a solution $u$ of a homogeneous problem, the positive eigenvalues of its Hessian control the negative ones and vice versa. Geometrically, the positive and negative curvatures of the graph of $u$ control each other. A way to define this precisely is by imposing the following condition on  $F$,
\begin{align*}
\cM^-(D^2(u-v)) \leq F(D^2u) - F(D^2v) \leq \cM^+(D^2(u-v)),
\end{align*}
where we are using the notation from \cite{CC}.

When we have in mind equations of order $\s \in (0,2)$, we may want to use $\cL_0$, the family of all linear operators $L$ which are comparable to the fractional laplacian of order $\s$ (to be defined), in order to define the ellipticity of $I$ as $\cM^-_{\cL_0}(u-v) \leq Iu - Iv \leq \cM^+_{\cL_0}(u-v)$. This family is however too big and in order to get further regularity we need to impose further assumptions on the linear operators. By this reason we give a definition of ellipticity which is more general.

\begin{definition}\label{ellipticoperator}
Let $\cL$ be a class of linear integro differential operators. We say that a fully non linear operator $I$ is elliptic with respect to the class $\cL$ if 
\begin{align}\label{ellipticity}
\cM^-_{\cL}(u-v) \leq Iu-Iv \leq \cM^+_{\cL}(u-v).
\end{align}
\end{definition}

Going back to the definition of $\cL_0 = \cL_0(\L,\s)$ ($\L\geq 1$). The precise condition for $L$ to be in $\cL_0$ with kernel $K$ is the following one,
\begin{align}
(2-\s)\frac{\L^{-1}}{|y|^{n+\s}}\leq K(y) \leq (2-\s)\frac{\L}{|y|^{n+\s}}.
\end{align}

In this family the extremal operators take the explicit form
\begin{align*}
\cM^+_{\cL_0} v(x,t) :&= \sup_{L\in\cL_0}(Lv)(x,t)\\
 &= (2-\s)\int\limits_{\R^n}\frac{\L\d^+(v,x,t;y)-\L^{-1}\d^-(v,x,t;y)}{|y|^{n+\s}}dy,\\
\cM^-_{\cL_0} v(x,t) :&= \inf_{L\in\cL_0}(Lv)(x,t)\\
& = (2-\s)\int\limits_{\R^n}\frac{\L^{-1}\d^+(v,x,t;y)-\L\d^-(v,x,t;y)}{|y|^{n+\s}}dy.
\end{align*}
Where $\d^\pm$ denote the positive and negative parts of $\d$ ($\d = \d^+ - \d^-$).

The factors $(2-\s)$ become important as $\s \to 2^-$ as they will allow us to recover second order differential operators.

H\"older regularity for the spatial gradient of $u$ requires ellipticity with respect to a smaller class. Given $\r_0>0$ we define $\cL_1 = \cL_1(\s,\L,\r_0) \ss \cL_0(\s,\L)$ by the family of operators $L \in \cL_1$ with kernel $K$, such that
\begin{align}\label{L1}
\int_{\R^n\sm B_{\r_0}}\frac{|K(y)-K(y-h)|}{|h|}dy \leq \L,
\end{align}
for every $|h| \leq \r_0/2$. It is sufficient that $|DK|\leq \L/(1+|y|^{n+\s}))$ for \eqref{L1} to hold. An important property of this stronger condition is that the parabolic equations associated with them remain invariant under scaling. We will not bother to give this family a specific notation because we will not use it in this work.


\subsection{Viscosity solutions.}

For viscosity solutions of an equation $u_t-Iu=f$ we always assume a minimum requirement of continuity for $u$. Here we denote the space of upper semicontinuous functions in $\bar\Omega\times[-T,0]$, always with respect to the parabolic topology, by $USC(\bar\Omega\times[-T,0])$. Similarly, $LSC(\bar\Omega\times[-T,0])$ denotes the space of lower semicontinuous functions in $\bar\Omega\times[-T,0]$.

With respect to the time derivative, it is natural for the parabolic topology to consider only the values of $u$ towards the past. In this sense
\begin{align*}
u_{t^-}(x,t) = \lim_{h\to0^+}\frac{u(x,t)-u(x,t-h)}{h}.
\end{align*}

\begin{definition}\label{viscosity}
A function $u \in USC(\bar\Omega\times[-T,0])$ ($u \in LSC(\bar\Omega\times[-T,0])$), is said to be a sub solution (super solution) to $u_t - Iu=f$, and we write $u_t - Iu \leq f$ ($u_t - Iu \geq f$), if every time $(v,B_r(x)\times(t-\t,t]) \in S$ touches $u$ by above (below) at $(x,t)$, i.e.
\begin{itemize}
\item[(i)] $v(x,t)=u(x,t)$,
\item[(ii)] $v(y,s)>u(y,s)$ ($\varphi(y,s)<u(y,s)$) for every $(y,s)\in B_r(x)\times(t-\t,t]\sm\{(x,t)\}$,
\end{itemize}
then $v_{t^-}(x,t) - Iv(x,t) \leq f(x,t)$ ($v_{t^-}(x,t) - Iv(x,t) \geq f(x,t)$).
\end{definition}

An equivalent definition holds if instead of using parabolic second order polynomials as test functions we use test functions $\varphi$ with less regularity around the contact point. This is important when we want to prove the maximum principle by means of and inf and sup convolutions. We omit it here and just assume that the maximum principle for viscosity solutions holds. The ideas for the proof of this result are standard and can be found in \cite{CC} or in the appendix of \cite{S2} for the non local case with $\s=1$.

The following example illustrates the importance of having test functions in $C(a,b;L^1(\w))$. Consider $u(x,t) = \chi_{E\times\{0\}}$ where $E \subset\subset \R^n\sm \bar B_1$. In the domain $B_1\times(-1,0)$ the function $u$ satisfies $u_{t^-} + (-\D)^{\s/2}u = 0$ in the classical sense. When $t=0$ the equation is not satisfied any more, $u_{t^-}(x,0)$ is still zero in $B_1$ but $(-\D)^{\s/2}u(x,0)$ becomes strictly positive in $B_1$. If we consider now the same equation in the viscosity sense, $u$ is a solution even when $t=0$. The restriction for the test functions to be in $C(-\t,0;L^1(\w))$ (in the case the contact occurs at $t=0$) implies that such test function will not be able to see that $u$ has a jump at $t=0$.


\subsection{Qualitative properties}

Most of the qualitative behavior of solutions of fully non linear non local operators, as considered by us, have been already proven in \cite{C1} or in the appendix of \cite{S2}. Here we state some of the results already known, that we will need to use later on. The first lemma was proven in \cite{C1} and says that for a super solution regularity by below implies that the operator can be evaluated in the classical way. The following three results are the expected maximum and comparison principles which can be proven as in \cite{S2}. The last theorem regards to the existence of viscosity solutions of the Dirichlet problem by Perron's method, here we show that by using and appropriated barrier the solution achieves the boundary values in a continuous way.

\begin{lemma}\label{touch}
Let $I$ be a elliptic operator with respect to $\cL_0$ and $f$ a continuous function. If we have a super solution, $u_t - Iu = f$ in $\W\times(-T,0]$ and $\varphi$ is a $C^2$ function that touches $u$ from below at a point $(x_0,t_0)$, then $Iu(x_0,t_0)$ is defined in the classical sense and $\varphi_{t^-}(x_0,t_0) - Iu(x_0) \geq f(x_0,t_0)$.
\end{lemma}

\begin{theorem}[Equation for the difference of solutions]\label{Eqdiff}
Let $I$ be a continuous elliptic operator with respect to $\cL_0$ and $f$ and $g$ continuous functions. Given $u$ and $v$ such that $u_t - Iu \leq f$ and $v_t - Iv \geq g$ hold in $\W\times(-T,0]$ in the viscosity sense, then $(u - v)_t - \cM^+_{\cL}(u - v) \leq f-g$ also holds in $\W\times(-T,0]$ in the viscosity sense.
\end{theorem}

\begin{theorem}[Maximum principle]
Let $u$ be a viscosity super solution of
\begin{align*}
 u_t - \cM^-_{\cL_0}u \geq 0 \text{ in } \W\times(-T,0].
\end{align*}
Then
\begin{align*}
\inf_{\bar \W \times [-T,0]}u = \inf_{((\R^n \sm \W)\times(-T,0]) \cup (\R^n\times\{-T\})}u.
\end{align*}
\end{theorem}

\begin{corollary}[Comparison principle]
Let $I$ be a continuous elliptic operator with respect to $\cL_0$, $u$ be a viscosity sub solution and $v$ be a viscosity super solution of
\begin{align*}
 w_t - Iw = f \text{ in } \W\times(-T,0].
\end{align*}
Then $u \leq v$ in $((\R^n \sm \W)\times(-T,0]) \cup (\R^n\times\{-T\})$ implies $u\leq v$ in $\W\times(-T,0]$.
\end{corollary}

Existence and uniqueness of a solution in the viscosity sense follows from the comparison principle by using Perron's method. The additional ingredient we need is a barrier that guarantees that the boundary and initial values are attained in a continuous way.

\begin{lemma}\label{barrier}
Let $\s\in(0,2)$. There exists a non negative function $\psi:\R^n\times(-\8,0]\to\R$ such that:
\begin{enumerate}
 \item $\psi = 0$ in $B_1\times\{0\}$
 \item $\psi_t-\cM^+_{\cL_0(\s)}\psi \geq 0$ in $(\R^n\sm B_1) \times(-\8,0]$,
 \item $\psi \geq 1$ in $(\R^n\times(-\8,0])\sm(B_2\times[-\k,0])$,
\end{enumerate}
for some $\k$ universal.
\end{lemma}

\begin{proof}
 Let $\varphi = \varphi(x)$ be the one from corollary 3.2 in \cite{C1} and
\begin{align*}
 \k^{-1} = \inf_{B_2\sm B_1}|\cM^+\varphi|\wedge1.
\end{align*}
Then $\psi = (\varphi - \k^{-1} t) \wedge 1$ satisfy all the requirements.
\end{proof}

\begin{theorem}[Existence]
Let $\s\in(0,2)$, $\W$ be a smooth domain, $I$ a continuous elliptic operator with respect to $\cL_0$ and $f$ and $g$ bounded, continuous functions. The Dirichlet problem, 
\begin{align*}
 u_t - Iu &= f \text{ in } \W\times(-T,0],\\
 u &= g \text{ in } ((\R^n \sm \W)\times(-T,0]) \cup (\R^n\times\{-T\}),
\end{align*}
has a unique viscosity solution $u$.
\end{theorem}

\begin{remark}
The boundary data $g$ only needs to be continuous at the points in $((\R^n \sm \W)\times(-T,0]) \cup (\R^n\times\{-T\})$ with respect to the parabolic topology. With respect to smoothness of the domain, we only require that $\W$ satisfies the exterior ball condition.\end{remark}

\begin{proof}
Let $u$ be the solution obtained by Perron's method,
\begin{align*}
u(x,t) = \inf\{v(x,t): &v_t - Iv \geq f \text{ in } \W\times(-T,0]\\
&v \geq g \text{ in } ((\R^n \sm \W)\times(-T,0]) \cup (\R^n\times\{-T\})\}.
\end{align*}
It can be shown that $u \in C(\bar\W\times[-T,0])$, solves $u_t - Iu = f$ in $\W\times(-T,0]$ in the viscosity sense and $u \geq g$ in $((\R^n \sm \W)\times(-T,0]) \cup (\R^n\times\{-T\})$, see \cite{CIL}. We will show now that $u$ attains the initial and boundary values by comparison with appropriated barriers.

Let's see the case of initial values first. Let $b :\R^n \to [0,1]$ a smooth bump function such that $\supp(1-b) = B_1$ and $b(0) = 0$. The function $\psi(y,s) = b(y) + \|\cM^+_{\cL_0}b\|_\8 s$ satisfies $\psi_t - \cM^+_{\cL_0}\psi \geq 0$. Let $(x,-T) \in \bar \W\times\{-T\}$ and $\e>0$ fixed. By the continuity of $g$, there exists $\d>0$ such that $|g(x,-T)-g(y,s)|\leq \e$, given that $|x-y|+ |T+s| \leq 2\d$. Consider the following barrier,
\begin{align*}
\b(y,s) = g(x,-T) + \e + 2\|g\|_\8 \3 \psi \1\frac{y-x}{\d},\frac{s+T}{\d^\s}\2 + \frac{s+T}{\d}\4.
\end{align*}
$\b$ it is constructed such that $\b_t - \cM^+_{\cL_0}\b \geq 0$. Let's see that $\b \geq g$ in $((\R^n\sm \W)\times(-T,0]) \cup (\W\times\{-T\})$ and therefore $g(x,-T) \leq u(x,-T) \leq \b(x,-T) = g(x,-T)+\e$ according to the definition of $u$. If $|x-y|+|T+s| \leq 2\d$, then $\b(y,s) \geq g(x,-T) + \e \geq g(y,s)$. If $|x-y|+|T+s| \geq 2\d$ then either $|x-y| \geq \d$ and then $b((y-x)/\d) = 1$ or $(T+s) \geq \d$ and then $(T+s)/\d \geq 1$, in any case
\begin{align*}
 \b(y,s) \geq -\|g\|_\8 + 2\|g\|_\8 \3 \psi \1\frac{y-x}{\d},\frac{T+s}{\d^\s}\2 + \frac{T+s}{\d}\4 \geq g(y,s).
\end{align*}
After having that $u(x,-T) \in [g(x,-T),g(x,-T)+\e]$ we use that $\e$ is arbitrary to conclude that $u(x,-T) = g(x,-T)$.

Let's consider now the case of boundary values. Let $(x,t) \in \p\W\times(-T,0]$ and $\e>0$ fixed. Let $\d>0$ such that $|g(x,t) - g(y,s)|\leq \e$, given that $|y-x| \leq \d$ and $s \in [t-\k\d,t]$, $\k$ is the one from Lemma \ref{barrier}. By making $\d$ even smaller we can also assume that the ball $B_{\d/4}(x-(\d/4)n)$ touches $\W$ by outside with $n$ its normal vector. Let $\psi$ the function from Lemma \ref{barrier}. The barrier,
\begin{align*}
\b(y,s) = g(x,t) + \e + 2\|g\|_\8\psi\1\frac{y-(x-(\d/4)n))4}{\d},s-t\2
\end{align*}
is constructed such that $\b_t - \cM^+_{\cL_0}\b \geq 0$. It also remain above $g$ in $((\R^n\sm \W)\times(-T,t]) \cup (\W\times\{-T\})$. If $|y-x| \leq \d$ and $s \in [t-\k\d,t]$ then $\b(y,s) \geq g(x,t) + \e \geq g(y,s)$. If $(y,s)$ is outside the cylinder $\bar B_\d(x)\times[t-\k\d,t]$, then it is also outside the cylinder $\bar B_{\d/2}(x-(\d/4)n)\times [t-\k\d,t]$, then $\psi(\frac{y-(x-(\d/4)n))4}{\d},s-t) \geq 1$ and $\b(y,s) \geq -\|g\|_\8 + 2\|g\|_\8 \geq g(y,s)$. Then we conclude as before that $u(x,t) = g(x,t)$.
\end{proof}


\section{Partial ABP Estimate}\label{SectionABP}

The classic ABP theorem says the following. If $u$ satisfies $u_t - \cM^-u \geq -f$ in $B_1\times(-1,0]$ with $u \geq 0$ in $\p B_1\times(-1,0] \cup B_1\times\{-1\}$ then,
\begin{align*}
\inf_{B_1\times(-1,0]} u^- \leq c\1\iint_{\{u=\G\}}(f^+)^{n+1}dxdt\2^{\frac{1}{n+1}},
\end{align*}
where the domain of integration $\{u=\G\}$ is the contact set of $u$ with its parabolic convex envelope $\G$.

In the non local case there is no hope to obtain a similar result by integrating only over $\{u=\G\}$. In fact, consider the function $u(x,t) = (|x|^\s-1)\chi_{B_2}(x)$. The contact set in this case has zero measure, however there is a constant $C \geq 0$ such that $u_t + (-\D)^\s u \geq -C$ holds in $B_1\times(-1,0]$.

To sort out this difficulty, we consider the set where $u$ is between $\G$ and $\G + M$, for some positive and universal $M$. The theorem we prove in this section is the following one.

\begin{theorem}\label{ABP}
Let $f \in C([-1,0])$ positive and depending only on the time variable, $\rho_0>0$ such that $1/2 + 9\sqrt n 2^{-1/(2-\s)}\r_0 < 2$ and let $u$ satisfying
\begin{align*}
u_t - \cM^-_{\cL_0} u &\geq -f\chi_{B_{1/2}} \text{ in } B_2\times(-2,0],\\
u &\geq 0  \text{ in } ((\R^n \sm B_1)\times[-2,0]) \cup (\R^n \times[-2,-1]\}),\\
\sup_{B_1\times(-1,0]}u^- &= 1.
\end{align*}
Then,
\begin{align*}
c \leq \int_{-1}^0 f(t)^{n+1}|\{u(\cdot,t) < \G(\cdot,t) + C4^{-1/(2-\s)}f(t)\}\cap B_{1/2 + 9\sqrt n 2^{-1/(2-\s)}\r_0}| dt
\end{align*}
for some constants $c$ and $C$ depending only on $n$, $\L$, $\s_0$ and $\r_0$.
\end{theorem}

\begin{remark}
The domain of the equation $B_2\times(-2,0]$ can be reduced to $B_{1+\e}\times(-1,0]$ and just indicates that we are going to need some room in the following proofs.
\end{remark}

\begin{remark}
The radius $\r_0$ will be a fixed universal constant in future sections. For this reason it may be noticed that we call also universal constants to some quantities that may also depend on $\r_0$.
\end{remark}

The following is a more general version of the same theorem.

\begin{theorem}\label{ABP2}
Let $\W \ss \R^n$ a bounded domain and $\W_2 \subset\subset \W_1 \subset\subset \W_0 \subset\subset \W$. Let $f \in C([-1,0])$ positive and depending only on the time variable, $d>0$ and let $u$ such that
\begin{align*}
u_t - \cM^-_{\cL_0} u &\geq -f\chi_{\W_2} \text{ in } \W\times(-d^\s,0],\\
u &\geq 0  \text{ in } ((\R^n \sm \W_0)\times(-d^\s,0]) \cup (\R^n \times\{-d^\s\}).
\end{align*}
Then
\begin{align*}
\sup_{\W_0\times(-d^\s,0]}u^- \leq C\1\int_{-d^\s}^0 f(t)^{n+1}|\{u < \G + C4^{-1/(2-\s)}f\}\cap \W_1| dt\2^{1/(n+1)}
\end{align*}
for some constant $C$, depending only on $n$, $\L$, $\s_0$, $\W_2$, $\W_1$, $\W_0$ and $\W$.
\end{theorem}

The idea of the proof is to cover $\{u = \G\}$ with a disjoint sequence of rectangles $K_j = Q_j \times I_j$, where $Q_j$ is a cube in space and $I_j$ is an interval in time such that:
\begin{enumerate}
\item The measure of the union of the covering is at least a fix constant.
\item In a dilation $\tilde K_j$ of $K_j$, $u$ is most of the time between $\G$ and $\G + M$.
\end{enumerate}


\subsection{Preliminaries}

Here we will fix the notation that we will carry on for the next results.

From $u$ we construct the following auxiliary functions which allow us to get important information for $u$. We enumerate them in the following list:
\begin{enumerate}
\item Let $(x_0,t_0) \in B_1\times(-1,0]$ such that $\sup_{B_1\times(-1,0]}u^- = |u(x_0,t_0)| = 1$.
\item Let $\bar u(x,t) = \inf_{s\in [-1,t]}u(x,s)$.
\item Let $\G(\cdot,t)$ be the convex envelope of $\bar u(\cdot,t)$ supported in $B_3$. Notice that $\G$ is convex in space and non increasing in time. This is what we also call a parabolic convex function.
\item Let $\p\G(x,t) \ss \R^n$ be the set of spatial subdifferential of $\G(\cdot,t)$ at $(x,t)$. Notice that for every $t\in[-1,0]$, $\p\G(B_3,t) = \p\G(B_1,t)$, therefore every subdifferential $p$ of $\G$ can be assume to be the slope of some supporting plane, namely $x\to p\cdot(x-x_0) + h$, to the graph of $\bar u(\cdot,t)$.
\item Let $h(\cdot,t):\p \G(B_1,t) \to \R$ be the Legendre transform of $\G(\cdot, t)$ centered at $x_0$,
\begin{align*}
h(p,t) = \sup\{h:p\cdot (x-x_0) + h \leq \G(x,t) \text{ for all } x \in B_3\},
\end{align*}
or equivalently,
\begin{align*}
h(p,t) = \sup\{h:&p\cdot (x-x_0) + h \leq \bar u(x,t) \text{ for all } x \in B_1,\\
                 &p\cdot (x-x_0) + h \leq 0           \text{ for all } x \in B_3\}.
\end{align*}
\item Let $\Phi = (\p\G(x,t),h(\p\G(x,t),t))$.
\end{enumerate}

This first lemma give some of the basic properties of the previously defined functions.

\begin{lemma}\label{Deltah}
Let $u$, $f$, $x_0$, $\G$, $h$ and $\Phi$ as defined above and consider also $\D t>0$. Then the following properties hold:
\begin{enumerate}
\item The domain of $h(\cdot,t)$ is non decreasing in time. i.e. $\p\G(B_1,t) \ss \p\G(B_1,t+\D t)$.
\item $h$ is non increasing in time.
\item The function $h$ restricted to $\{\G = u\}$ is Lipschitz in time. Specifically, for $(x_1,t_1) \in \{\G = u\}$ and $p_1 \in \p\G(x_1,t_1)$
\begin{align*}
\D h := h(p_1,t_1+\D t) - h(p_1,t_1) \geq -2\|f\|_{L^\8([t_1,t_1+\D t])}\D t.
\end{align*}
\end{enumerate}
\end{lemma}

\begin{proof}
The first two properties are consequences of the monotonicity of $\G$. If at time $t$, the plane $x \to p\cdot(x-x_0) + h$ is a supporting plane for the graph of $\G(\cdot,t)$ then at time $t + \D t$ it crosses or touches the graph of $\G(\cdot,t+\D t) \leq \G(\cdot,t)$, therefore by lowering $h$ we can find a supporting plane for $\G(\cdot,t+\D t)$ with the same slope $p$.

In order to see the next property notice that for every $p\in\G(\cdot,t)$ and $x\in B_3$,
\begin{align*}
p\cdot(x-x_0) + h(p,t+\D t) \leq \G(x,t + \D t) \leq \G(x,t).
\end{align*}
This makes $h(p,t+\D t)$ an admissible candidate in the definition of $h(p,t)$, and therefore $h(p,t) \geq h(p,t+\D t)$.

For the second part, notice first that $p_1 \in \p\G(B_1,t_1+\D t)$ because of the first property, therefore $\D h$ is well defined. We also have proved that $\D h \leq 0$ in the second property, so assume that $\D h < 0$  and consider the following test function,
\begin{align*}
v(x,s) = \1p_1\cdot(x-x_0) + h(p_1,t_1) + \frac{\D h}{2\D t}(s-t_1)\2\chi_{B_1}(x).
\end{align*}
The infimum of $u-v$ in $\R^n\times[-1,t_1+\D t]$ is strictly negative and attained at some point $(x_2,t_2) \in B_1 \times (t_1,t_1 + \D t]$. Indeed, the plane $x\to p_1\cdot(x-x_0) + h(p_1,t_1)$ crosses the graph of $u$ in $B_1\times(t_1,t_1 + \D t]$, otherwise $\D h$ would not be strictly negative.

We have that $v_t(x_2,t_2) = \D h/2\D t$ and $\d(v(\cdot,t_2),x_2;y) \geq 0$. This is immediate if $x_2\pm y$ lie both inside $B_1$ or both lie outside $B_1$. If $x_2+y \in B_1$ we use that $u^-(\cdot,t_2) \leq 1$ in order to see that $|v(x_2+y,t_2)|$ is at most $2|v(x_2,t_2)|$. Because $u^- \leq 1$, the plane $x\to p_1\cdot(x-x_0) + h(p_1,t_1)$ is above $-1$ at some point in $B_1$, lets recall also that the same plane is below zero in $B_3$. It tells us that its slope is at most $1/2$ and then $v(x_2,t_2) - v(x_2+y,t_2) = -p_1\cdot y \leq |p_1||y| \leq 1 \leq -v(x_2,t_2)$. Then, if $x_2-y$ is outside $B_1$, $\d(v(\cdot,t_1),x_1;y) = (v(x_1+y,t_1) - 2v(x_1,t_1)) + v(x_1-y,t_1) \geq 0$.

We conclude by the monotonicity of $d/dt - \cM^-_{\cL_0}$,
\begin{align*}
-\|f\|_{L^\8([t_1,t_1+\D t])} &\leq (u_t - \cM^-_{\cL_0}u)(x_2,t_2),\\
&\leq (v_t - \cM^-_{\cL_0}v)(x_2,t_2),\\
&\leq \frac{\D h}{2\D t}.
\end{align*}
\end{proof}


\subsection{Configurations of the covering pieces}

In the following lemmas we will study how the solution $u$ detach from $\G$ around the contact set. In order to keep the statements as simple as possible we will describe here some recurrent geometric configurations and fix the notation for them. For $u$ satisfying the hypothesis of Theorem \ref{ABP}, $k \in \N$, $\D t,\r_0\in(0,1)$, $(x_1,t_1) \in B_1\times(-1,0]$ we consider:
\begin{enumerate}
\item $R_i = B_{r_i}(x_1) \sm B_{r_{i+1}}(x_1)$ for $r_i = 2^{-i}2^{-1/(2-\s)}\r_0$,
\item $S_i = R_i \times [t_1-\D t,t_1-\D t/2]$.
\end{enumerate}
Eventually $(x_1,t_1)$ will be fixed to be in the contact set, $k$ will also be fixed of the order of $1/(2-\s)$.


\begin{lemma}\label{IterLemma}
Let $u$ satisfy the hypothesis of Theorem \ref{ABP} and for $k \in \N$, $\D t\in(0,1)$, $(x_1,t_1) \in B_1\times(-1,0]$ let $r_i$ and $S_i$ as defined above. Given $(p_1,h_1) \in \R^{n+1}$, $M>0$ and $\m \in (0,1)$ such that
\begin{enumerate}
\item $\G(x,t) \geq p_1\cdot(x-x_0)+h_1$ for $(x,t)\in B_3\times(t_1-\D t,t_1]$,
\item For every $i=0,1,\ldots,k-1$,
\begin{align}\label{hyp}
\frac{|\{u - (p_1\cdot(x-x_0)+h_1)\geq \|f^+\|_{L^\8([t_1-\D t,t_1])}Mr^2_i\} \cap S_i|}{|S_i|} \geq \m,
\end{align}
\end{enumerate}
then for $(x,t)\in B_{r_{k+1}}(x_1) \times [t_1-\D t/2,t_1] $ we have
\begin{align*}
 u - (p_1\cdot(x-x_0)+h_1) \geq \|f^+\|_{L^\8([t_1-\D t,t_1])}\D t
\end{align*}
if $\D t \in (0,r_k^2)$ and $M\mu(r_0^{2-\s}-r_k^{2-\s}) \geq K$, for some constant $K$ independent of $\s$.
\end{lemma}

\begin{corollary}\label{Ring}
Let $u$ satisfy the hypothesis of Theorem \ref{ABP} and for $k \in \N$, $\D t\in(0,1)$, $(x_1,t_1) \in B_1\times(-1,0]$ let $r_i$ and $S_i$ as defined above. For every $\s<2$ there is some $k \sim 1/(2-\s)$ such that if,
\begin{enumerate}
\item $\D t \in (0,r_k^2)$,
\item $(x_1,t_1) \in \{\G = u\}$,
\item $p_1 \in \p\G(x_1,t_1)$,
\end{enumerate}
then there is some sufficiently small radius $r\in(0,r_0)$, so that for $S = (B_r(x_1)\sm B_{r/2}(x_1))\times(t_1-\D t,t_1-\D t/2]$ the following holds for every $M>0$,
\begin{align*}
 \frac{|\{u - (p_1\cdot(x-x_0)+h(p_1,t_1))\geq \|f^+\|_{L^\8([t_1-\D t,t_1])}Mr^2\} \cap S|}{|S|} \leq \frac{2K}{r_0^2}M^{-1}.
\end{align*}
\end{corollary}


\begin{proof}[Proof of Lemma \ref{IterLemma}]
Let $\b = \b_0(x-x_1/r_k) \in [0,1]$ be a smooth bump function such that $supp\ \b_0 = B_{3/4}$ and $\b_0 = 1$ in $B_{1/2}$. Moreover we can choose $\b_0$ such that $\cM^-_{\cL_0}\b_0 \geq 0$ if $\b_0(x) \leq \b_1$ for some positive constant $\b_1$. We want to use a test function in $B_1 \times [t_1-\D t,t_1]$ of the form
\begin{align*}
 v(x,t) &= P(x) + m(t)\b(x) - \|f^+\|_{L^\8([t_1-\D t,t_1])}(t - (t_1 - \D t)),\\
 P(x)   &= (p_1\cdot(x-x_0)+h_1)\chi_{B_1}(x),
\end{align*}
such that $m(t_1-\D t) = 0$ and $m \geq 2\|f^+\|_{L^\8([t_1-\D t,t_1])}\D t$ for $t \in [t_1-\D t/2,t_1]$.

Assume by contradiction that,
\begin{align*}
 \inf_{B_{r_{k+1}(x_1)} \times [t_1-\D t/2,t_1]}(u - P) < \|f^+\|_{L^\8([t_1-\D t,t_1])}\D t.
\end{align*}
Then for some $(x_2,t_2) \in B_{3r_k/4}(x_1) \times [t_1-\D t,t_1]$,
\begin{align*}
 \inf_{\R^n\times [t_1-\D t/2,t_1]} (u-v) = (u-v)(x_2,t_2) < 0.
\end{align*}
Notice that $B_{3r_k/4}(x_1) \times [t_1-\D t,t_1]$ is contained in the domain of the equation $B_2\times(-2,0]$ if $\D t < 1$ and $r_0 < 1$. We use Lemma \ref{touch} in order to do the following computations on $u$ at the contact point $(x_2,t_2)$. We also use that $\d(u-v,x_2,t_2) = \d^+(u-v,x_2,t_2)$ and $\d(P,x_2,t_2;y) \geq 0$ for every $y\in\R^n$, as in the proof of Lemma \ref{Deltah}, 
\begin{align*}
&-m'(t_2)\b(x_2)  + m(t_2)r^{-\s}_k\cM^-_{\cL_0}\b_0(x_2-x_0),\\
&\leq (u_t - \cM^-_{\cL_0}u)(x_2,t_2) - (v_t - \cM^-_{\cL_0}v)(x_2,t_2),\\
&\leq (u-v)_t(x_2,t_2) - \cM^-_{\cL_0}(u-v)(x_2,t_2),\\
&\leq -\L^{-1}(2-\s)\int_{\bigcup_{i=0}^{k-1} R_i \times \{t_2\}}\frac{\d^+(u-v,x_2,t_2;y)}{|y|^{n+\s}}dy,\\
&= -\L^{-1}(2-\s)\int_{\bigcup_{i=0}^{k-1} R_i \times \{t_2\}}\frac{\d^+(u-m\b,x_2,t_2;y)}{|y|^{n+\s}}dy.
\end{align*}
If $x_2 + y \in R_i$, then $|y| \sim r_i$, because $x_2 \in B_{3r_k/4}(x_1)$. Also $\d^+(u-m\b) \geq \|f^+\|_{L^\8([t_1-\D t,t_1])}Mr_i^2$ every time $x_2+y \in G_i(t_2)$, where
\begin{align*}
G_i(t) = \{u - (p\cdot(x-x_0)+h_1) \geq \|f^+\|_{L^\8([t_1-\D t,t_1])}Mr_i^2\} \cap (R_i\times\{t\}).
\end{align*}
Therefore we obtain, for some constant $C_0$ depending only on $\L^{-1}$ and the dimension,
\begin{align}\label{eq1}
 -m'\b + mr_k^{-\s}\cM^-_{\cL_0}\b_0 \leq -C_0(2-\s)\|f^+\|_{L^\8([t_1-\D t,t_1])}M\sum_{i=0}^{k-1}\frac{|G_i(t_0)|}{r_i^n}r_i^{2-\s}.
\end{align}
As in \cite{S1}, if $m$ satisfies:
\begin{align*}
m'(t) &= c_1(2-\s)\|f^+\|_{L^\8([t_1-\D t,t_1])}M\sum_{i=0}^{k-1}\frac{|G_i(t)|}{r_i^n}r_i^{2-\s} - C_2r_k^{-\s}m(t),\\
m(t_1-\D t) &= 0,
\end{align*}
with constants $c_1$ sufficiently small and $C_2$ sufficiently large then we get a contradiction. Indeed, if $\b_0(x_2-x_1/r_k) \leq \b_1$ then $\cM^-_{\cL_0}\b_0 \geq 0$ and we see that $0 < c_1 \leq C_0$ implies a contradiction by substituting $m'$ in \eqref{eq1}. Otherwise, if $\b_0(x_2-x_1/r_k) > \b_1$, then we also get a contradiction in \eqref{eq1} if $C_2 \geq \|\cM^-_{\cL_0}\b_0\|_\8/\b_1$.

We finally need to check that $m \geq 2\|f^+\|_{L^\8([t_1-\D t,t_1])}\D t$ in $[t_1-\D t/2,t_1]$ from the hypothesis of the lemma. We have an explicit formula for $m$,
\begin{align*}
 m(t) = c_1(2-\s)\|f^+\|_{L^\8([t_1-\D t,t_1])}M\sum_{i=0}^{k-1}r_i^{2-\s}\int_{t_1-\D t}^{t}\frac{|G_i(s)|}{r_i^n}e^{-C_2r_k^{-\s}(t-s)}ds.
\end{align*}
Using the hypothesis \eqref{hyp} of the lemma, for $t\geq t_1-\D t/2$, we get
\begin{align*}
 m \geq C(2-\s)\|f^+\|_{L^\8([t_1-\D t,t_1])}M\m \D t\frac{r_0^{2-\s}-r_k^{2-\s}}{1-2^{\s-2}}e^{-C_2r_k^{-\s}\D t}.
\end{align*}
The quotient $(2-\s)/(1-2^{\s-2})$ is bounded away from zero by a universal constant when $\s \in [0,2]$. Also $e^{-C_2r_k^{-\s}\D t} \geq e^{-C_2}$ if $\D t \leq r_k^\s$. Finally $m \geq 2\|f^+\|_{L^\8([t_1-\D t,t_1])}\D t$ is satisfied if $M\mu(r_0^{2-\s}-r_k^{2-\s}) \geq K$ for some $K$ independent of $\s$.
\end{proof}


The following is a geometric lemma that can be applied to any parabolic convex function. As a reminder, we say that $\G(x,t)$ is a parabolic convex function if it is convex in the variable $x \in \R^n$ and non increasing in the variable $t \in \R$.

\begin{lemma}\label{tapaarriba}
Let $\G:B_3\times[-\D t,0]\to \R$ parabolic convex function such that
\begin{align*}
\frac{|\{\G \geq M\} \cap (B_r\sm B_{r/2})\times[-\D t,-\D t/2]|}{|(B_r\sm B_{r/2})\times[-\D t,-\D t/2]|} \leq \e_0.
\end{align*}
Then $\G \leq M$ in $B_{r/2}\times[-\D t/2,0]$ if $\e_0$ is sufficiently small, depending only on $n$.
\end{lemma}

\begin{proof}
By the convexity of $\G$ we can assume that its maximum $N$ over $B_{r/2}\times[-\D t/2,0]$, is attained at $(r/2e_1, -\D t/2)$. Therefore $\G \geq N$ in $A = \{(x,t) \in (B_r\sm B_{r/2})\times[-\D t,-\D t/2]: x\cdot e_1 > r/2\}$. Therefore, if $\e_0$ is smaller than $|A|/|(B_r\sm B_{r/2})\times[-\D t,-\D t/2]|$, we obtain that $N$ is necessarily smaller or equal than $M$ in $B_{r/2}\times[-\D t/2,0]$.
\end{proof}


By applying Lemma \ref{Deltah} and the previous lemma to $\G(x,t) - (p_1\cdot(x-x_0) + h(p_1,t_1))$ with all the hypothesis and conclusions of Corollary \ref{Ring} we obtain the following result.

\begin{corollary}[Flatness of $\G$]\label{Flat}
Let $u$ satisfy the hypothesis of Theorem \ref{ABP} and $(x_1,t_1)$, $p_1$ and $r$ as in Corollary \ref{Ring}. There exist a universal constant $M>0$ such that for every $(x,t) \in F = B_{r/2}(x_1)\times[t_1-\D t/2,\max\{t_1+\D t/2, 0\}]$
\begin{align*}
-2\|f^+\|_{L^\8([t_1,\max\{t_1+\D t, 0\}])}r^2 &\leq \G(x,t) - (p_1\cdot(x-x_0) + h(p_1,t_1)),\\
&\leq M\|f^+\|_{L^\8([t_1-\D t,t_1])}r^2.
\end{align*}
\end{corollary}

The previous corollary seems still insufficient to control
\begin{align*}
 \frac{|\Phi(B_{r/4}(x_1) \times [t_1-\D t/2,t_1+\D t/2])|}{|B_{r/4}(x_1) \times [t_1-\D t/2,t_1+\D t/2]|}.
\end{align*}
The flatness property takes care of the $n$-dimensional size of $\p\G(B_{r/4}(x_1) \times \{t\})$ by using the geometry of the convex function $\G(\cdot,t)$ ($t\in[t_1-\D t/2,t_1+\D t/2]$). Note also that the image of $\Phi(\cdot,t) = (\p\G(\cdot,t), h(\p\G(\cdot,t),t))$ is the graph of $h(\cdot,t)$, for which we use again the properties in Lemma \ref{Deltah}.

\begin{corollary}\label{measure}
Let $u$ satisfy the hypothesis of Theorem \ref{ABP} and $(x_1,t_1)$, $p_1$ and $r$ as in Corollary \ref{Ring}. There exist a universal constant $C>0$ such that for $K = B_{r/4}(x_1)\times[t_1-\D t/2,\max\{t_1+\D t/2,0\}]$ we have 
\begin{align*}
 \frac{|\Phi(K)|}{|K|} \leq C\|f^+\|_{L^\8([t_1-\D t,\max\{t_1+\D t,0\}])}^{n+1}.
\end{align*}
\end{corollary}

\begin{proof}
We do the proof for $t_1 + \D t\leq 0$ in order to avoid the difficulties that would arise if $K$ goes beyond $t=0$. In this case the proof does not differ to much from the one we present but makes the proof more technical.
 
As a consequence of Corollary \ref{Flat} we have that $\p\G(B_{r/4}(x_1) \times \{t_1+\D t/2\}) \ss B_{Cr\|f^+\|_{L^\8([t_1-\D t,t_1])}}(p_1)$ and then
\begin{align*}
 \frac{|\p\G(B_{r/4}(x_1) \times \{t_1+\D t/2\}|}{|B_{r/4}(x_1)|} \leq C\|f^+\|_{L^\8([t_1-\D t,t_1])}^n.
\end{align*}
By Lemma \ref{Deltah},
\begin{align*}
 \Phi(B_{r/4}(x_1) \times [t_1-\D t/2,t_1+\D t/2]) \ss Cylinder,
\end{align*}
where
\begin{align*}
 Cylinder = \{(p,h):&p \in \p\G(B_{r/4}(x_1) \times \{t_1+\D t/2\}),\\
                    &h \in [h(p,t),h(p,t) + 2\D t\|f^+\|_{L^\8([t_1,t_1+\D t])}]\}.
\end{align*}
Finally the measure of $Cylinder$ is controlled by its base times the height.
\end{proof}


\subsection{Covering of the contact set}

We state now a weak version of the ABP estimate. The result consists in finding a covering of the contact set where the solution does not separate too much from the convex envelope in a given fraction of the union of the covering.

\begin{lemma}\label{SABP}
Let $u$ satisfy the hypothesis of Theorem \ref{ABP}. There exists a finite family of disjoint rectangles $\{K_j=Q_j\times I_j\}$, where $Q_j \ss \R^n$ is an open cube with diameter $d_j\leq r_0/4$ and $I_j = (-(l_j+1)\D t/2,-l_j\D t/2)$ with $l_j$ non negative integer, such that:
\begin{enumerate}
 \item $K_j \cap \{u = \G\} \neq \emptyset$,
 \item $\bigcup_j \bar K_j \supseteq \{u=\G\}$,
 \item \label{a} $\G$ is between two planes in $Q_j \times I_j$ which are separated by a distance $C\|f\|_{L^\8(I_j)}d_j^2$,
 \item \label{b} $|\Phi(K_j)| \leq C\|f\|_{L^\8(I_j)}^{n+1}|K_j|$,
 \item \label{c} $|\{u < \G + C\|f\|_{L^\8(I_j)}d_j^2\}\cap \tilde K_j| \geq (1-\e_0)|\tilde K_j|$, where $\tilde K_j = 16\sqrt n Q_j\times [-(l_j+3)\D t/2,-l_j\D t/2]$.
\end{enumerate}
\end{lemma}

\begin{proof}
Fix a slice $B_1\times I_l$ and cover it by a tiling of the form $\{Q\times I_l\}$ where $Q$ have diameter $r_0/4$. Discard all of those rectangles that do not intersect $\{u=\G\}$. Whenever $Q\times I_l$ does not satisfy (\ref{a}), (\ref{b}) or (\ref{c}), we split $Q$ into $2^n$ cubes $Q'$ of half diameter and discard all of the rectangles $Q'\times I_l$ whose closure does not intersect $\{u=\G\}$. We need to prove that eventually all rectangles satisfy (\ref{a}), (\ref{b}) and (\ref{c}) and therefore the process finishes after a finite number of steps. In fact we will show that it will finish before $k \sim 1/(2-\s)$ iterations.
 
As before, in order to avoid technical difficulties, we assume that $l \geq 1$. Let $Q_1 \times I_l \supseteq Q_2 \times I_l \supseteq \ldots \supseteq Q_k \times I_l \ni (x_1,t_1)$ such that $(x_1,t_1)\in \{u = \G\}$ and let's see that at least one of those rectangles satisfy all the properties (\ref{a}), (\ref{b}) and (\ref{c}). From Lemmas \ref{Flat} and \ref{measure}, there is some radius $r \in [r_k,r_0]$ and some subdifferential $p_1\in\p\G(x_1,t_1)$ such that the following are true:
\begin{enumerate}
\item $|\G - (p_1\cdot(x-x_0) + h(p_1,t_1))| \leq C\|f\|_{L^\8(I_l)}r^2$ in $F = B_{r/4}\times [t_1-\D t/2, t_1+\D t/2]$,
\item $|\Phi(K)| \leq C\|f\|_{L^\8(I_l)}^{n+1}|K|$ for $K = B_{r/4}\times[t_1-\D t/2, t_1+\D t/2]$,
\item $|\{u - (p_1\cdot(x-x_0)+h(p_1,t_1)) \geq \|f\|_{L^\8(I_l)}Mr^2\} \cap S| \leq \e_0|S|$ for $S=B_r\times[t_1-\D t, t_1-\D t/2]$.
\end{enumerate}

There is one of the rectangles $Q_j \times I_l$, with $\diam(Q_j)=d$, such that $r/8 \leq d < r/4$. Therefore $Q_j \times I_l \ss K(x_1,t_1)$ and conditions (\ref{a}) and (\ref{b}) from the lemma are verified. To check (\ref{c}), notice that $S \ss 16\sqrt n Q_j \times [-(l+3)\D t/2,-l\D t/2] = \tilde K_j$, and that the volumes of $S$, $K_j$ and $\tilde K_j$ are comparable, hence
\begin{align*}
&|\{u < \G + \|f\|_{L^\8(I_l)}Md_j^2\}\cap \tilde K_j|,\\
&\geq |\{u < p_1\cdot(y-x_1)+h(p_1,t_1)+\|f\|_{L^\8(I_l)}Md_j^2\}\cap S|,\\
&\geq (1-\e_0)|S|.
\end{align*}
This is how $\m$ is chosen and this concludes the proof.
\end{proof}


\subsection{Proof of Theorem \ref{ABP}}

We have as in \cite{T} that
\begin{align*}
\Phi(\{\G = u\}) \supseteq Cone = \{(p,h): h\in[-1,0], |h| > 4|p|\}.
\end{align*}
The inclusion follows because for every $(p,h) \in Cone$ the plane $x\to p\cdot(x-x_0) + h$ can be brought from $t=-1$ towards the future until it hits the graph of $u$ (and also the graphs of $\bar u$ and $\G$) for the first time.

Therefore for some universal constants, 
\begin{align*}
C &\leq |\Phi(\cup_j K_j)| \leq \sum_j |\Phi(K_j)| \leq C\sum_j \|f\|_{L^\8(I_j)}^{n+1}|K_j|
\end{align*}
We group now the previous sum in each interval $J_l = (-(l+1)\D t/2, -l\D t/2)$,
\begin{align*}
C \leq \sum_l \|f\|_{L^\8(J_l)}^{n+1}\sum_{I_j = J_l}|K_j| \leq \sum_l \|f\|_{L^\8(J_l)}^{n+1} \left|\bigcup_{I_j = J_l}K_j\right|
\end{align*}
By Besicovitch, we can take a sub set of $\{\tilde K_j\}_{I_j = J_l}$ (denoted by the same) with the finite intersection property and still covering $\cup_{I_j = J_l}K_j$ such that
\begin{align*}
\sum_{I_j = J_l} |\tilde K_j| &\leq C\sum_{I_j = J_l} |\{u < \G + C4^{-1/(2-\s)}\|f\|_{L^\8(J_l)}\} \cap \tilde K_j|,\\
&\leq C\left|\{u < \G + C4^{-1/(2-\s)}\|f\|_{L^\8(J_l)}\} \cap \1\cup_{I_j = J_l}\tilde K_j\2\right|.
\end{align*}
Notice that the contact set $\{u=\G\}$ can only occur where $f\chi_{B_{\r_0}}$ is positive. This implies that $\1\cup_{I_j = J_l}\tilde K_j\2 \ss B_{9\sqrt n \r_0}$ and then we have the following Riemann sum which is now independent of the covering,
\begin{align*}
c \leq \sum_l \|f\|_{L^\8(J_l)}^{n+1}\left|\{u < \G + C4^{-1/(2-\s)}\|f\|_{L^\8(J_l)}\} \cap B_{9\sqrt n \r_0}\times \tilde J_l\right|
\end{align*}
where $\tilde J_l = (-(l+3)\D t/2, -l\D t/2)$. Now we just have to send $\D t$ to zero to conclude the theorem.

In \cite{C1} the partial ABP involves a Riemann sum of $|f|^n$ which gets refined in the limit, when $\s$ goes to 2, and allows to recover the classic ABP. The estimate presented here is weaker, assuming that the right hand side of the equation is $f(x,t)$ we notice that in Lemma \ref{Deltah} we need to take a global $L^\8$ norm in space and not just around the contact point. All the other proofs work fine in this sense. Our proofs can recover the following consequence of the classical ABP as $\s$ goes to two,
\begin{align*}
\sup_{B_1\times(-1,0]} u^- \leq C \1\int_{\{\Gamma = u\}}\sup_{y \in B_1}(f^+)^{n+1}dt\2^{1/n+1}.
\end{align*}


\section{Point Estimate}\label{SPE}

We are interested now in proving a point estimate that will allow us to control the oscillation of the solution. Our goal is the following theorem. (Recall that $f \in C[-1,0]$ and non negative).

\begin{theorem}[Point Estimate]\label{PE}
Let $\s_0 \in (0,2)$ and $\s \in (\s_0,2)$. Suppose $u$ satisfies 
\begin{align*}
u_t - \cM^-_{\cL_0(\s)}u &\geq -f(t) \text{ in } B_1 \times (-1,0],\\
u &\geq 0 \text{ in } \R^n\times[-1,0].
\end{align*}
Then, for every $s\geq 0$, 
\begin{align*}
\frac{|\{u > s\} \cap B_{1/2} \times [-1,-1/2]|}{|B_{1/2} \times [-1,-1/2]|} \leq C\1\inf_{B_{1/2} \times [-1/2,0]}u + \|f^+\|_{L^\8([-1,0])}\2^\e s^{-\e},
\end{align*}
for some constants $\e$, $C$ depending only on $n$, $\L$ and $\s_0$.
\end{theorem}

The proof of Theorem \ref{PE} is done by induction as in \cite{W}. The idea is to get a control of the measure of the set where $u$ is bigger than a universal constant and then being able to reproduce the estimate at every scale.

\subsection{Initial configurations}


\begin{lemma}[Special Function]\label{SpecialFunction}
Let $\s_0 \in (0,2)$, $\s \in (\s_0,2)$. There is a function $p(x,t)\in C(\R^n\times [0,80])$ and a constant $C>1$, such that for any $\s \in (\s_0,2)$,
\begin{align*}
p_{t^-}-\cM^-_{\cL_0(\s)}p &\leq -1+C\chi_{B_{1/4}\times (0,1]} \text{ in } B_{4\sqrt n}\times (0,80],\\
p(x,t) &\leq 0 \text{ in } ((\R^n\sm B_{2\sqrt n})\times(0,80]) \cup (\R^n\times\{0\}),\\
p(x,t) &> 2 \text{ in } Q_3\times[1,80].
\end{align*}
The function $p$ will also be $C^{1,1}$ in the space variable and $C^1$ in the time variable (with respect to the parabolic topology), so that the computation of the equation is done in the classical sense.
\end{lemma}
 
\begin{proof}
Consider
\begin{align*}
f(x) = \begin{cases}
|x|^{-p} &\text{ in $\R^n \sm B_\d$},\\
q &\text{ in $B_\d$},
\end{cases}
\end{align*}
where $q$ is a quadratic polynomial chosen so that $f$ is $C^{1,1}$ across $\p B_\d$. From Section 9 in \cite{C1} we know that it satisfies $\cM^-_{\cL_0(\s)}f > 0$ in $\R^n \sm B_{1/4}$ for some sufficiently large $p>0$ and some sufficiently small $\d\in(0,1/4)$, independently of $\s \in (\s_0,2)$. By multiplying $f$ by a sufficiently large constant we can also assume that
\begin{align*}
\cM^-_{\cL_0(\s)}f &\geq 1 - C_0\chi_{B_{1/4}} \text{ in $B_{4\sqrt n}$},\\
|Df(x)\cdot x| &\leq C_1 \text{ in $\R^n$},\\
f &< C_2 \text{ in $\R^n$}.
\end{align*}

Consider also a continuous function $m(t) \geq 0$ such that for some $\t\in(0,1)$ to be fixed,
\begin{align*}
m(t) &= t^{1/2} \text{ in  $[0,\t]$},\\
m(t) &= \t^{1/2}e^{-\frac{C_1+C_0}{C_2\t}(t-\t)} \text{ in $(\t,80]$},
\end{align*}

Let $\tilde p(x,t) = m(t)f(y)$, for $y = t^{-2/\s_0}x$. Notice that $\tilde p(\cdot,0) \equiv 0$ defines $\tilde p$ continuously up to time zero.

In the region $B_{4\sqrt n} \sm B_{1/4}\times(0,\t]$ we have that,
\begin{align*}
(\tilde p_t - \cM^-_{\cL_0}\tilde p)(x,t) &= \frac{f(y)}{2t^{1/2}} - (2/\s_0) t^{-1/2}Df(C_2y)\cdot y - t^{1/2-2\s/\s_0}\cM^-_{\cL_0}f(y),\\
&< (C_2/2 + 2C_1/\s_0)t^{-1/2} - t^{-3/2},
\end{align*}
Therefore, by choosing $\t \leq 1/(C_2/2 + 2C_1/\s_0)$ we make $p_t - \cM^-_{\cL_0}p < 0$ in $B_{4\sqrt n} \sm B_{1/4}\times[0,\t]$.

In the region $B_{4\sqrt n}\times(\t,80]$ we have that,
\begin{align*}
(\tilde p_t - \cM^-_{\cL_0}\tilde p)(x,t) &= m'(t)f(y) - \frac{m(t)}{t}Df(C_2y)\cdot y - \frac{m(t)}{t^{\s/\s_0}}\cM^-_{\cL_0}f(y),\\
&< m'(t)C_2 + m(t)(C_1 + C_0)/\t,
\end{align*}
which is zero by the construction of $m$ in $(\t,80]$.

Finally, we define $p = A(\tilde p - B)^+$ with $B\geq 0$ chosen such that $\tilde p - B \leq 0$ in $((\R^n\sm B_{2\sqrt n})\times(0,80]) \cup (\R^n\times\{0\})$ and $A\geq1$ chosen such that $p > 2$ in $Q_3\times[1,80]$ and $p_t-\cM^-_{\cL_0(\s)}p \leq -1$ in $B_{4\sqrt n}\times(0,80]$.
\end{proof}

We are in shape now to prove a first control of the distribution.

\begin{lemma}[Base configuration] \label{base}
Let $\s_0 \in (0,2)$, $\s \in (\s_0,2)$, and $u$ a function such that
\begin{align*}
u_t-\cM^-_{\cL_0(\s)}u &\geq -1 \text{ in } B_{4\sqrt n}\times (0,80],\\
u &\geq 0 \text{ in } \R^n\times[0,80],\\
\inf_{Q_3\times[1,80]}u &\leq 1,
\end{align*}
Then
\begin{align*}
\left|\3u > M_0\4\cap Q_1 \times [0,1]\right| \leq \m_0|Q_1 \times [0,1]|,
\end{align*}
for some universal constants $\m_0 \in (0,1)$ and $M_0>1$.
\end{lemma}
 
\begin{proof}
Let $v=u-p$, where $p$ was constructed in the previous lemma. The graph of $v$ goes below $-1$ at some point in $Q_3\times[1,80]$, stays non negative in $((\R^n \sm Q_3)\times(1,80]) \cup (\R^n\times\{0\})$ and satisfies
\begin{align*}
v_t - \cM^-_{\cL_0} v \geq (u_t - \cM^-_{\cL_0}u) - (p_t - \cM^-_{\cL_0}p) \geq -C\chi_{B_{1/4}\times(0,1]}.
\end{align*}

We apply now a Theorem \ref{ABP2} to $v$ with $\W = B_{4\sqrt n}$, $\W_0 = Q_3$, $\W_2 = B_{1/4}$ and $\W_1 = Q_1$,
\begin{align*}
c \leq \int_0^1|\{v < \G_v + C\}\cap B_{1/2}|dt \leq |\{u < p + C\} \cap Q_1 \times [0,1]|.
\end{align*}
We just choose $\m_0 = 1-c$ and $M_0 = \sup_{Q_1 \times [0,1]} p + C$ to conclude.
\end{proof}

We will need the following corollaries of the previous lemma. The first one will be necessary for the particular dyadic decomposition we will introduce in the next section. The second one iterates $m$ times Corollary \ref{Basetau}.

\begin{corollary}[Flexible configuration] \label{Basetau}
Let $\s_0 \in (0,2)$, $\s \in (\s_0,2)$, $\t\in[1,8]$ and $u$ a function such that
\begin{align*}
u_t-\cM^-_{\cL_0(\s)}u &\geq -1 \text{ in } B_{4\sqrt n}\times (0,(3^\s+1)\t],\\
u &\geq 0 \text{ in } \R^n\times[0,(3^\s+1)\t],\\
\inf_{Q_3\times[\t,(3^\s+1)\t]}u &\leq 1.
\end{align*}
Then for $\m_1 = \frac{7+\m_0}{8}$,
\begin{align*}
\left|\3u > M_0\4\cap Q_1 \times [0,\t]\right| \leq \m_1|Q_1 \times [0,\t]|,
\end{align*}
for $\m_0 \in (0,1)$ and $M_0>1$ as in Lemma \ref{base}.
\end{corollary}

\begin{proof}
Notice that $Q_3\times[\t,(3^\s+1)\t] \ss Q_3\times[1,80]$ therefore if $u$ goes below 1 in $Q_3\times[\t,(3^\s+1)\t]$ then $\left|\3u > M_0\4\cap Q_1 \times [0,1]\right| \leq \m_0|Q_1 \times [0,1]|$ and
\begin{align*}
\left|\3u > M_0\4\cap Q_1 \times [0,\t]\right| &\leq \m_0|Q_1 \times [0,1]| + \left|\3u > M_0\4\cap Q_1 \times [1,\t]\right|,\\
&\leq \frac{7 + \m_0}{8}|Q_1 \times [0,\t]|.
\end{align*}
\end{proof}

\begin{corollary}[Iteration]\label{Iteration}
Let $\s_0 \in (0,2)$, $\s \in (\s_0,2)$, $\t\in[1,8]$, $k\geq 1$ a natural number, $d_i = \frac{3^{\s(i+1)}-1}{3^\s-1}$ and $u$ a function such that
\begin{align*}
u_t-\cM^-_{\cL_0(\s)}u &\geq -1 \text{ in } B_{2\sqrt n 3^k}\times \left(0,d_k\t\right],\\
u &\geq 0 \text{ in } \R^n\times\left[0,d_k\t\right],\\
\inf_{\cup_{i=1}^{k-1} Q_{3^{\s i}}\times\left[d_i\t,d_{i+1}\t\right]}u &\leq 1.
\end{align*}
Then
\begin{align*}
\left|\3u > M_0^k\4\cap Q_1 \times [0,\t]\right| \leq \m_1|Q_1 \times [0,\t]|,
\end{align*}
for $\m_1 \in (0,1)$ and $M_0>1$ as in Corollary \ref{Basetau}.
\end{corollary}

\begin{proof}
Just apply Corollary \ref{Basetau}, rescaled, $k$ times.
\end{proof}

\subsection{A Calder\'on - Zygmund Lemma}

The purpose of a Calder\'on - Zygmund type Lemma is to find a cover of a given set $A$ with dyadic boxes that capture a fraction of $A$ around a given $\m_1 \in (0,1)$.

The boxes are chosen from a dyadic decomposition of $Q_1\times[0,1]$ that almost preserve the scaling of the equation. When $\s$ is either 1, $\log_2 3$ or 2 the decomposition can be made preserving the scaling of the equation by dividing by 2, 3 or 4 in time respectively. The following algorithm for general $\s$ was communicated to us by Luis Caffarelli.

Initially we consider $Q_1\times[0,1]$, split $Q_1$ into $2^n$ congruent cubes and take the $2^n$ possible cartesian products with $[0,1]$ to form the new dyadic boxes. In each step we consider one of the cubes $Q_r(x_0)\times[t_0,t_0+r^\s\t]$ and always divide $Q_r(x_0)$ in $2^n$ congruent cubes; with respect to the time interval we do the following:
\begin{enumerate}
\item If $\t<2$ then we do not subdivide $[t_0,t_0+r^\s\t]$.
\item If $2\leq\t<4$ then we subdivide $[t_0,t_0+r^\s\t]$ in 2 congruent intervals.
\item If $4\leq\t$ then we subdivide $[t_0,t_0+r^\s\t]$ in 4 congruent intervals.
\end{enumerate}
Finally we take the cartesian product to form the new generation of dyadic boxes from $Q_r(x_0)\times[t_0,t_0+r^\s\t]$.

This procedure verifies that if $\t\in[1,8]$, then the boxes that $Q_r(x_0)\times[t_0,t_0+r^\s\t]$ generates have dimensions $r/2$ (in space) and $(r/2)^\s\t'$ (in time) for some $\t'\in[1,8]$. To prove it, just consider each of the cases.

Given two dyadic boxes $K$ and $\tilde K$ we say that $\tilde K$ is a predecessor of $K$ if $K$ is one of the boxes obtained from the decomposition of $\tilde K$.

The following lemma follows as the one in chapter 4 of \cite{CC} with the difference that the Lebesgue decomposition theorem is applied to rectangles of dimensions $\r$ and $\r^\s$ instead of the standard cubes.

\begin{lemma}\label{CZ1}
Let $A \ss Q_1\times[0,1]$ and $\m_1\in(0,1)$, such that $|A| \leq \m_1|Q_1\times[0,1]|$. Then there exists a set of disjoint dyadic boxes $\{K_j\}$ such that:
\begin{enumerate}
\item $|\cup_j K_j \sm A| = 0$,
\item $|A \cap K_j| > \m_1|K_j|$,
\item $|A \cap \tilde K_j| \leq \m_1|\tilde K_j|$.
\end{enumerate}
\end{lemma}

\begin{proof}
Starting with $Q_1\times[0,1]$, we subdivide the dyadic boxes (with the previous algorithm) that capture a fraction of $A$ smaller or equal to $\m_1$ and select those boxes $\{K_j\}$ that capture a fraction bigger than $\m_1$. Initially $Q_1\times[0,1]$ captures a fraction of $A$ smaller or equal to $\m_1$, therefore we know that $Q_1\times[0,1]$ is subdivided and $\tilde K_j \ss Q_1\times[0,1]$.

This process selects a family of disjoint boxes $\{K_j\}$ that satisfy 2 and 3.

To verify 1 we use the Lebesgue differentiation theorem. For each $(x,t) \in \cup_j K_j \sm A$ there exist a family of dyadic boxes $\{K_i^{(x,t)} = Q_{r_i}(x_j)\times[t_i,t_i+r_i^\s\t_i]\}_{i\geq1}$ such that,
\begin{enumerate}
\item $(x,t) \in K_i^{(x,t)}$,
\item $r_i \to 0$ as $i\to\8$ and $\t_i\in[0,8]$,
\item $|A \cap K_i^{(x,t)}| \leq \m_1|K_i^{(x,t)}|$.
\end{enumerate}
From $K_i^{(x,t)} = Q_{r_i}(x_j)\times[t_i,t_i+r^\s\t]$ we construct a box with a scale $\s$, $\bar K_i^{(x,t)} = Q_{\r_i}(x_j)\times[t_i,t_i+\r_i^\s] \supseteq K_i^{(x,t)}$ such that $\r_i = r_i\t_i^{1/\s}$. They satisfy instead,
\begin{enumerate}
\item $(x,t) \in \bar K_i^{(x,t)}$,
\item $\r_i \to 0$ as $i\to\8$,
\item $|A \cap \bar K_i^{(x,t)}| \leq \bar\m_1|\bar K_i^{(x,t)}|$ with $\bar\m_1 = \frac{(8^{\s_0}-1) + \m_0}{8^{\s_0}}<1$.
\end{enumerate}
Then we can apply a modified version of the Lebesgue differentiation theorem to conclude that $|\cup_j K_j \sm A| = 0$. See for instance Exercise 3 in Chapter 7 of \cite{WZ}.
\end{proof}

This lemma however can not be applied in our situation directly. The results of the previous section say that if $u$ goes below 1 in some region in the future then we can control the distribution in the past; but the predecessor $\tilde K_j$ might give no information of what happens with $u$ in the future. For this reason we need to consider shifts in time of $\tilde K_j$. For a given cube $K = Q \times [t_0,t_0+r]$ and a natural number $m\geq1$ let $K^m = Q\times[t_0+r,t_0+(m+1)r]$.

The following lemma is proven as in section 3 of \cite{W}.

\begin{lemma}\label{CZ}
Let $A \ss Q_1\times[0,1]$ and $\m_1\in(0,1)$, such that $|A| \leq \m_1|Q_1\times[0,1]|$. Then there exists a set of disjoint dyadic boxes $\{K_j\}$ such that:
\begin{enumerate}
\item $|\cup_j K_j \sm A| = 0$,
\item $|A \cap K_j| > \m_1|K_j|$,
\item $|A| \leq \frac{(m+1)\m_1}{m}|\cup_j (\tilde K_j)^m|$.
\end{enumerate}
\end{lemma}

\begin{proof}
Select the same covering $\{K_j\}$ of $A$ from Lemma \ref{CZ1}. Consider a disjoint sub covering of $\{\tilde K_j\}$ which also covers $A$ in measure (denoted by the same). Then
\begin{align*}
|A| \leq \sum_j|A \cap \tilde K_j| \leq \m_1|\cup_j \tilde K_j| \leq \m_1|\cup_j (\tilde K_j \cup (\tilde K_j)^m)|.
\end{align*}

Now we show that
\begin{align}\label{czeq1}
|\cup_j (\tilde K_j \cup (\tilde K_j)^m)| \leq \frac{m+1}{m}|\cup_j (\tilde K_j)^m|.
\end{align}
Let $E$ be the interior of $\cup_j (\tilde K_j)^m$ and consider the open bounded sets of the real line $E_x = \{t\in\R: (x,t) \in E\}$. Take the decomposition of $E_x$ into a countable set of disjoint open intervals $E_x = \cup_j I_x^j$ and for every $I_x^j = (a,b)$ take $TI_x^j = (a-\frac{1}{m+1}(b-a),b)$. Define also $TE_x = \cup_j TI_x^j$ and $TE = \cup_x TE_x$. By Fubini,
\begin{align*}
|TE| &= \int_{\R^n} |TE_x|,\\
&= \int_{\R^n} |\cup_j TE_x^j|,\\
&\leq \frac{m+1}{m}\int_{\R^n} |E_x^j|,\\
&\leq \int_{\R^n} \frac{m+1}{m}|E_x|,\\
&= \frac{m+1}{m}|E|,\\
&\leq \frac{m+1}{m}|\cup_j (\tilde K_j)^m|.
\end{align*}
On the other hand, the interior of $\cup_j (\tilde K_j \cup (\tilde K_j)^m)$ is contained in $TE$, by the definition of $T$ and the definition of the stacks. But $\cup_j (\tilde K_j \cup (\tilde K_j)^m)$ and its interior have the same measure and therefore we obtain \eqref{czeq1} and conclude the proof.
\end{proof}

By combining Corollary \ref{Iteration} with Lemma \ref{CZ} we get the following result.

\begin{lemma}\label{stack}
Let $\s_0 \in (0,2)$, $\s \in (\s_0,2)$, $m \geq m_0$ and $u$ a function such that
\begin{align*}
u_t-\cM^-_{\cL_0(\s)}u &\geq -1 \text{ in } B_{2\sqrt n 3^m}\times(0,d_m],\\
u &\geq 0 \text{ in } \R^n\times[0,d_m]
\end{align*}
Then for any $i \geq 1$, the dyadic covering $\{K_j\}$ of $A = \3u > M_1^{i+1}\4\cap (Q_1 \times [0,1])$ with respect to the fraction $\m_1$ satisfies
\begin{align*}
\cup_j (\tilde K_j)^m \ss \3u > M_1^i\4\cap (Q_1 \times [0,d_m]),
\end{align*}
for $m_0$ universal, $M_1 = M_0^m$, and $\m_1 \in (0,1)$, $M_0>1$, $d_m$ as in Corollary \ref{Iteration}.
\end{lemma}

\begin{proof}
Notice first that $M_1 = M_0^m$ implies by Corollary \ref{Iteration} that
\begin{align*}
|A| \leq |\{u > M_0^m\}\cap Q_1\times[0,1]| \leq \m_1|Q_1\times[0,1]|,
\end{align*}
which allows us to apply Lemma \ref{CZ} to $A$.

By the construction of $(\tilde K_j)^m$ we know that $\cup_j (\tilde K_j)^m \ss Q_1 \times [0,d_m]$ for $m$ large enough. So we can assume by contradiction that there exists some $K_j=Q_r(x_0)\times[t_0,t_0+r^\s\t]$ such that
\begin{enumerate}
\item $\inf_{(\tilde K_j)^m} u \leq 1$,
\item $|A \cap K_j|>\m_1|K_j|$.
\end{enumerate}

The rescaling $(x,t) \to (r^{-1}(x-x_0),r^{-\s}(t-t_0))$ sends $K_j$ to $Q_1\times[0,\t]$; the stack $(\tilde K_j)^m$ to a subset of $\cup_{i=1}^m Q_{3^\s i}\times[d_i\t,d_{i+1}\t]$ if $m$ is large enough (at this moment we can fix $m_0$); and it sends the domain of the equation $B_{2\sqrt n 3^m}\times(0,d_m\t]$ to a even bigger domain.

The function $v(x,t) = \frac{u(rx+x_0,r^\s t+t_0)}{M_1^i}$ satisfies
\begin{align*}
v_t - \cM^-_{\cL_0}v &\geq -\frac{r^\s}{M_1^i} \geq -1 \text{ in } B_{2\sqrt n 3^m}\times(0,d_m\t],\\
v &\geq 0 \text{ in } \R^n\times(0,d_m\t],\\
\inf_{\cup_{i=1}^{m-1} Q_{3^\s i}\times[d_i\t,d_{i+1}\t]}v &\leq 1.
\end{align*}
By Corollary \ref{Iteration} we get,
\begin{align*}
\frac{\left|A \cap K_j\right|}{|K_j|} \leq \frac{\left|\3v > M_0^m\4\cap Q_1 \times [0,1]\right|}{|Q_1 \times [0,1]|} \leq \m_1.
\end{align*}
But this contradicts the construction of $K_j$.
\end{proof}


\subsection{Proof of Theorem \ref{PE}}

By combining the previous results we get the following discrete version of Theorem \ref{PE} which implies the proof of Theorem \ref{PE} by rescaling and a covering argument.

\begin{lemma}\label{Discrete}
Let $\s_0 \in (0,2)$, $\s\in(\s_0,2)$ and $u$ a function such that
\begin{align*}
u_t-\cM^-_{\cL_0}u &\geq -1 \text{ in } B_{2\sqrt n 3^m}\times(0,d_m],\\
u &\geq 0 \text{ in } \R^n\times[0,d_m],\\
\inf_{Q_1\times[C_0-1,C_0]}u &\leq 1,
\end{align*}
then for any $k \in \N$
\begin{align*}
 |\{u > M_2^k\}\cap Q_1\times[0,1]| \leq \m_3^k|Q_1\times[0,1]|,
\end{align*}
for some universal constants $M_2>1$, $\m_3\in(0,1)$, $C_0$, $m$; and $d_m$ as in Corollary \ref{Iteration}.
\end{lemma}

\begin{proof}
Let $C_0=\frac{8(6\sqrt n)^{\s_0}}{3^{\s_0}-1}+2$.

We proceed by induction. The case $k=1$ is the result of Corollary \ref{Iteration} with $\t=1$ if $C_0 \leq d_m$, $M_2 \geq M_0^m$ and $\m_3 \in (\m_0,1)$. Assume then it is true for some $k$ and lets prove it for $k+1$. It means that we are assuming that
\begin{align*}
|\{u>M_2^{k+1}\} \cap Q_1\times[0,1]| > \m_3^{k+1} \geq \m_3|\{u>M_2^k\}\cap Q_1\times[0,1]|.
\end{align*}

Let $\{K_j\}$ the dyadic disjoint covering of $A=\{u>M_2^{k+1}\} \cap Q_1\times[0,1]$ with respect to the fraction $\m_1$. Lets fix $m$ sufficiently large such that:
\begin{enumerate}
\item $\m_2 := \frac{(m+1)\m_1}{m} < 1$,
\item $d_m\geq C_0$,
\item $m \geq m_0$ with $m_0$ from Lemma \ref{stack}.
\end{enumerate}

Lemma \ref{CZ} and \ref{stack} tell us that for every piece $K_j$ of the covering,
\begin{align*}
|\{u>M_2^{k+1}\} \cap K_j| &> \m_1|K_j|,\\
|\{u>M_2^{k+1}\} \cap Q_1\times[0,1]| &\leq \m_2|\cup_j (\tilde K_j)^m|,\\
\cup_j (\tilde K_j)^m &\ss \{u>M_2^k\}\cap Q_1\times[0,d_m].
\end{align*}
Therefore,
\begin{align*}
\m_3^{k+1} &< |\{u>M_2^{k+1}\} \cap Q_1\times[0,1]|,\\
&\leq \m_2|\cup_j (\tilde K_j)^m|,\\
&= \m_2(|\cup_j (\tilde K_j)^m \cap Q_1\times[0,1]| + |\cup_j (\tilde K_j)^m\cap Q_1\times[1,d_m]|),\\
&\leq \m_2(|\{u>M_2^i\}\cap Q_1\times[0,1]| + |\cup_j (\tilde K_j)^m\cap Q_1\times[1,d_m]|),\\
&\leq \m_2(\m_3^k + |\cup_j (\tilde K_j)^m\cap Q_1\times[1,d_m]|).
\end{align*}
By fixing $\m_3 = (1+\m_2)/2 \in (0,1)$ we get for some constant $c$,
\begin{align*}
c\m_3^k \leq |\cup_j (\tilde K_j)^m\cap Q_1\times[1,d_m]|.
\end{align*}
Therefore there is a point $(x,t) \in \cup_j (\tilde K_j)^m$ with $t\geq c\m_3^k+1$. Hence there is a dyadic box $Q_r(x_0)\times[t_0,t_0+r^\s\t]$ in the covering with
\begin{enumerate}
\item[a.-] $4(m+1)r^\s\t \geq c\m_3^k$,
\item[b.-] $|\{u>M_2^{k+1}\} \cap Q_r(x_0)\times[t_0,t_0+r^\s\t]| > \m_1|Q_r(x_0)\times[t_0,t_0+r^\s\t]|$,
\end{enumerate}
The first equation implies $d_{Nk}r^\s\t \geq C_0$ for $N$ large enough, independent of $k$, $\s$ and $\t$.

In order to apply Corollary \ref{Iteration} (rescaled) we need to check at least that
\begin{align}\label{stackeq}
Q_1\times[C_0-1,C_0] \ss \cup_{i=1}^{k-1} Q_{3^{\s i}r}(x_0)\times\left[t_0 + d_ir^\s\t,t_0 + d_{i+1}r^\s\t\right]
\end{align}
By having chosen $C_0-2=\frac{8(6\sqrt n)^{\s_0}}{3^{\s_0}-1}$ we get that for $(x_0,t_0) \in Q_1\times[0,1]$ and $(x,t) \in Q_1\times[C_0-1,C_0]$, $t-t_0 \geq \frac{\t(6\sqrt n)^{\s_0}}{3^{\s_0}-1}$ and $|x-x_0| \leq \sqrt n$. Therefore
\begin{align*}
Q_1\times[C_0-1,C_0] \ss \{\t \leq t-t_0 \leq \t|6(x-x_0)|^\s/(3^\s-1)\}.
\end{align*}
Using the previous relation and that $d_{Nk}r^\s\t \geq C_0$ we verify \eqref{stackeq}. It allows us then to conclude that $|\{u>M_0^{Nk}\} \cap Q_r(x_0)\times[t_0,t_0+r^\s\t]| \leq \m_1|Q_r(x_0)\times[t_0,t_0+r^\s\t]|$ which would contradict
\begin{align*}
|\{u>M_2^{k+1}\} \cap Q_r(x_0)\times[t_0,t_0+r^\s\t]| > \m_1|Q_r(x_0)\times[t_0,t_0+r^\s\t]|
\end{align*}
if $M_2 \geq M_0^N$.
\end{proof}



\section{Regularity Results}\label{RR}
The purpose of this section is to prove that solutions of \eqref{eqbase} are regular.

H\"older regularity follows by proving a geometric decay of the oscillation of the solution. H\"older regularity for the spatial gradient holds if the equation is translation invariant and the operator is elliptic with respect to $\cL_1$.

H\"older regularity for the time derivative does not hold even for the fractional heat equation in a bounded domain if the boundary data is not sufficiently nice. Consider $u$, the solution of $u_{t^-}+(-\D)^\s u = 0$ in $B_1\times(-1,0]$, with initial value $u(\cdot,-1) \equiv 0$. The boundary data outside $B_1$ is equal to $\underbar u=(c(t+1/2) + \chi_{B_3 \sm B_2}(x))\chi_{[-1,2,0]}(t)$ where the constant $c>0$ is chosen small enough so that $\underbar u$ is a subsolution to the fractional heat equation in $B_1\times(-1,0]$. By the comparison principle, we have that $u \geq \underbar u$ in $B_1\times(-1,0]$. Also $u \equiv 0$ in $B_1\times[-1,-1/2]$ by uniqueness. This shows that the time derivative of $u$ have a jump at $t=-1/2$. This problem will be treated in a future work.



\begin{theorem}[H\"older regularity]\label{Calpha}
Let $\s_0 \in(0,2)$ and $\s\in(\s_0,2)$. Let $u\in C(\W\times(-1,0])$ such that it satisfies the following two inequalities in the viscosity sense with $C_0 \geq 0$,
\begin{align*}
u_t - \cM_{\cL_0(\s)}^- u &\geq -C_0 \text{  in $B_2\times(-1,0]$},\\
u_t - \cM_{\cL_0(\s)}^+ u &\leq C_0 \text{  in $B_2\times(-1,0]$}.
\end{align*}
Then there is some $\a \in (0,1)$ and $C>0$, depending only on $n$, $\L$ and $\s_0$, such that for every $(y,s),(x,t) \in B_{1/2}\times(-1/2,0]$
\begin{align*}
\frac{|u(y,s) - u(x,t)|}{(|x-y| + |t-s|^{1/\s})^\a} \leq C\1\|u\|_{L^\8(\bar B_2\times[-1,0])} + \|u\|_{C(-1,0;L^1(\w))} + C_0\2
\end{align*}
\end{theorem}

\begin{proof}
Assume without loss of generality that $(x,t) = (0,0)$ and also that $u \in [-1/2,1/2]$ in $B_2 \times [-1,0]$. A truncation of $u$ also satisfies a similar equation with a controlled right hand side. Let $v = u\chi_{B_2}$, then
\begin{align*}
v_t - \cM_{\cL_0}^- v &\geq -C_0 - C\|u\|_{L^1(\w)} \text{  in $B_1\times(-1,0]$},\\
v_t - \cM_{\cL_0}^+ v &\leq C_0 + C\|u\|_{L^1(\w)} \text{  in $B_1\times(-1,0]$}.
\end{align*}
Assume without loss of generality that $C_0 + C\|u\|_{C(-1,0;L^1(\w))} \leq \e_0$ for some $\e_0>0$ small enough to be fixed. The idea of the proof is to construct an increasing sequence $m_k$ and a decreasing sequence $M_k$, such that $M_k-m_k=2(1/2)^{\a k}$ and $v$ is trapped between $m_k$ and $M_k$ in $B_{(1/2)^k}\times[-(1/2)^{\s k},0]$.

Take initially $m_0=-1/2$ and $M_0=1/2$. Assume we have constructed the sequences up to some index $k$. We want to find now how to construct $m_{k+1}$ and $M_{k+1}$. Consider
\begin{align*}
w(x,t)=\frac{v((1/2)^kx, (1/2)^{\s k} t)-(m_k+M_k)/2}{(1/2)^{\a k}},
\end{align*}
so that $w \in [-1,1]$ in $B_1 \times [-1,0]$ and satisfies
\begin{align*}
w_t - \cM_{\cL_0}^- w &\geq -\e_02^{-(\s-\a)k} \geq -\e_0 \text{  in $B_1\times(-1,0]$},\\
w_t - \cM_{\cL_0}^+ w &\leq \e_02^{-(\s-\a)k} \leq \e_0 \text{  in $B_1\times(-1,0]$},
\end{align*}
given that $\a \in (0,\s_0)$.

In $B_{1/2}\times[-1/2,0]$, $w$ is between $-1$ and 1 and either is above or below zero at least $1/2$ of the measure of $B_{1/2}\times[-1,-1/2]$. Assume, without loss of generality, that
\begin{align*}
\frac{|\{w > 0\}\cap B_{1/2}\times[-1,-1/2]|}{|B_{1/2}\times[-1,-1/2]|} \geq \frac{1}{2}.
\end{align*}
Then we want to use Theorem \ref{PE} to show that $w+1$ raises from 0 in $B_{1/2}\times[-1/2,0]$. Still $w$ can have negative values outside $B_1$ which will not allow to apply such theorem. However, $w+1 \geq -2(|2x|^\a-1)$ outside $B_1$, so that $w^+$ satisfies in $B_{3/4}\times(-1,0]$,
\begin{align*}
(w+1)^+_t - \cM_{\cL_0}^- (w+1) \geq -\e_0 - \|2(|2x|^\a-1)\chi_{\R^n\sm B_1}\|_{L^1(\w)},
\end{align*}
By choosing $\a$ small enough we can get right hand sides of magnitude $2\e_0$.

By Theorem \ref{PE} we get that,
\begin{align*}
C\1\inf_{B_{1/2}\times[-1/2,0]}w+2\e_0\2\geq \frac{|B_{1/2}\times[-1/2,0]|}{2},
\end{align*}
Which for $\e_0$ small enough implies that $\inf_{B_{1/2}\times[-1/2,0]}w \geq \theta$ for some universal $\theta>0$.

Then in this case we can choose $M_{k+1} = M_k$ and $m_{k+1} = m_k + \theta(M_k-m_k)/2$. Notice that, $M_{k+1} - m_{k+1} = (1-\theta/2)2^{-\a k}$ which can be made smaller than $2^{-\a k}$ for $\a$ even smaller. In the case that 
\begin{align*}
\frac{|\{w > 0\}\cap B_{1/2}\times[-1,-1/2]|}{|B_{1/2}\times[-1,-1/2]|} \geq \frac{1}{2},
\end{align*} 
we arrive to the conclusion that $M_{k+1} = M_k - \theta(M_k-m_k)/2$ and $m_{k+1} = m_k$ satisfies all the inductive hypothesis which concludes the proof.
\end{proof}


\begin{theorem}[$C^{1,\a}$ Regularity for translation invariant operators]
Let $\s_0 \in(0,2)$, $\s\in(\s_0,2)$ and $f\in C([-1,0])$. There is $\r_0$ depending only on $n,\l, \L, \s_0$, such that if $I$ is an inf sup (or sup inf) type operator, translation invariant in space and elliptic with respect to $\cL_1(\s,\L,\r_0)$ and $u \in C(\bar B_1\times[-1,0])$ is a viscosity solution of the equation,
\begin{align*}
u_t - Iu = f(t) \text{ in $B_2\times(-1,0]$},
\end{align*}
then $u$ is $C^{1,\a}$ in space for some universal $\a \in (0,1)$. More precisely, there is a constant $C>0$, depending only on $n$, $\L$ and $\s$, such that for every $(x,t),(y,s) \in B_{1/4}\times(-1,0]$
\begin{align*}
\frac{|u_{x_i}(x,t) - u_{x_i}(y,s)|}{(|x-y| + |t-s|^{1/\s})^\a} &\leq C\1\|u\|_{L^\8(\bar B_2\times[-1,0])} + \|u\|_{C(-1,0;L^1(\w))}\right.\\
&\left. {} + \|f\|_{L^\8((-1,0])}\2, \text{ for $i=1,\ldots,n$}.
\end{align*}
\end{theorem}

\begin{proof}
Assume without loss of generality that 
\[\|u\|_{L^\8(\bar B_2\times[-1,0])} + \|f\|_{L^\8((-1,0])} \leq 1.
\]
By considering the cutoff of $u$ in $B_2$ we have, as in the previous proof, that we can assume also that $u \equiv 0$ in $\R^n \sm B_2$.

Let $\bar\a$ the H\"older exponent obtained by Theorem \ref{Calpha} and assume that it is not the reciprocal of an integer by making it smaller if necessary. Let $\d = 1/(4\lfloor1/\bar\a\rfloor)$. Fix a unit vector $e\in\R^n$, a number $h \in (0,\d/8)$ and, for $k=1,2,\ldots,\lfloor1/\bar\a\rfloor$, let $\eta^k$ be a smooth cut-off function supported in $B_{(3/4 - k\d) - \d/4}$ and equal to one in $B_{(3/4 - k\d) -\d/2}$. Define the following incremental quotients,
\begin{align*}
 w^{h,k}(x,t)   &= \frac{u(x+he,t) - u(x,t)}{|h|^{\bar\a k}},\\
 w_1^{h,k}(x,t) &= \frac{(\eta^k u)(x+he,t) - (\eta^k u)(x,t)}{|h|^{\bar\a k}},\\
 w_2^{h,k}(x,t) &= \frac{((1-\eta^k) u)(x+he,t) - ((1-\eta^k) u)(x,t)}{|h|^{\bar\a k}}.
\end{align*}
We will show that for every $k=0,1,\ldots,\lfloor1/\bar\a\rfloor-1$, if
\begin{align}\label{hypc1a}
\|w^{h,k}\|_{L^\8(B_{3/4-k\d}\times[-(3/4-k\d),0])} &\leq C(k)
\end{align}
then in $B_{3/4-(k+1)\d}\times[-(3/4-(k+1)\d),0]$ the following estimate holds
\begin{align}\label{conc1a}
\frac{|w^{h,k}(x,t) - w^{h,k}(y,s)|}{(|x-y| + |t-s|^{1/\s})^\a} &\leq C(k+1)
\end{align}
where $w^{h,0} = u$ and the constants $C(k)$ are independent of $h$.

When $(x,t) \in B_{(3/4 - k\d)-\d/8}\times[-((3/4 - k\d)-\d/8),0]$, $|w_1^{h,k}|$ is bounded above by the product rule and the hypothesis \eqref{hypc1a},
\begin{align*}
|w_1^{h,k}(x,t)| &\leq C(k) + \|\eta^k\|_{L^\8(B_{(3/4 - k\d)-\d/8}\times[-((3/4 - k\d)-\d/8),0])}.
\end{align*}
If $x \in (\R^n \sm B_{(3/4-k\d)-\d/8})\times[-((3/4 - k\d)-\d/8),0]$ then $w_1^{h,k}(x,t)$ just cancels.

By using that the equation is translation invariant we have that $u$ and $u(\cdot+he,\cdot)$ satisfy equations in the same ellipticity family. By Theorem \ref{Eqdiff}, $w^{h,k}$ satisfy the following inequalities in $B_1\times(-1,0]$ in the viscosity sense,
\begin{align*}
w^{h,k}_t - \cM^-_{\cL_1}w^{h,k} &\geq 0,\\
w^{h,k}_t - \cM^+_{\cL_1}w^{h,k} &\leq 0.
\end{align*}
The function $w_1^{h,k}$ satisfy a similar equation as $w^{h,k}$ in $B_{(3/4-k\d)- 3\d/4}\times(-1,0]$, the difference is on the right hand side introduced by the cutoff,
\begin{align*}
(w_1^{h,k})_t - \cM^-_{\cL_1}w_1^{h,k} &\geq -\cM^-_{\cL_1}w_2^{h,k},\\
(w_1^{h,k})_t - \cM^+_{\cL_1}w_1^{h,k} &\leq -\cM^+_{\cL_1}w_2^{h,k}.
\end{align*}

For $x \in B_{(3/4-k\d)- 3\d/4}$ the terms $|\cM^\pm_{\cL_1} w_2^h|$ are controlled by $\|u\|_\8 = 1$ by using that
\[
\int_{\R^n \sm B_{\r_0}} \frac{|K(y)-K(y-h)|}{|h|}dy \leq C \text{ every time $|h| < \frac{\rho_0}{2}$}.
\]
with $\r_0 = \d/8$. Indeed, for $L \in \cL_1$ with kernel $K$, $x \in B_{(3/4-k\d)- 3\d/4}$, $|y| \leq \d/8$ we have that $w_2^{h,k}(x+y) = 0$ and by the product rule
\begin{align*}
|Lw_2^{h,k}(x)| &= \left|\int w_2^{h,k}(x+y,t)K(y)dy\right|,\\
&= \left|\int_{\R^n\sm B_{\d/8}} \frac{(1-\eta^k)u(x+y+h)-(1-\eta^k)u(x+y)}{|h|^{\bar\a k}}K(y)dy\right|,\\
&= \left|\int_{\R^n\sm B_{\d/8}} (1-\eta^k)u(x+y)|h|^{1-\bar\a k}\frac{K(y)-K(y-h)}{|h|}dy\right|,\\
&\leq C.
\end{align*}

We get then the equations for $w_1^{h,k}$ in $B_{(3/4-k\d)- 3\d/4}\times(-((3/4-k\d)-3\d/4),0]$
\begin{align*}
(w_1^{h,k})_t - \cM^-_{\cL_1}w_1^{h,k} &\geq -C,\\
(w_1^{h,k})_t - \cM^+_{\cL_1}w_1^{h,k} &\leq C.
\end{align*}
By applying Theorem \ref{Calpha} to $w_1^{h,k}$ from $B_{(3/4-k\d)-3\d/4}\times(-((3/4-k\d)-3\d/4),0]$ to $B_{3/4-(k+1)\d}\times(-(3/4-(k+1)\d),0]$ we conclude that for a constant $C(k+1)$ independent of $h$,
\begin{align*}
\frac{|w_1^{h,k}(x,t) - w_1^{h,k}(y,s)|}{(|x-y| + |t-s|^{1/\s})^\a} \leq C(k+1).
\end{align*}
Which is equivalent to \eqref{conc1a}.

By Lemma 5.3 in \cite{CC} we get that \eqref{conc1a} implies that $w^{h,k+1}$ is also bounded by a constant independent of $h$. Therefore we can apply this procedure up to obtaining that $u_x$ is bounded in $B_{1/2}\times[-1/2,0]$. Finally, by applying the previous argument one more time to the Lipschitz quotient we conclude the theorem.
\end{proof}

{\bf Acknowledgment.} The authors would like to thank Luis Caffarelli for proposing the problem and for various useful discussions. The authors would also like to thank Luis Silvestre for pointing out several typos in an earlier version of this paper.


\end{document}